\documentclass[a4paper]{article}

\usepackage{amsthm}
\usepackage{graphicx}
\usepackage{comment}
\usepackage{color}
\usepackage{enumitem}
\usepackage[top=2cm, bottom=2cm, left=2cm, right=2cm]{geometry}

\usepackage{hyperref} 
\usepackage{xcolor}
\hypersetup{
    colorlinks=true,
    citecolor=blue,
    linkcolor=blue,
    urlcolor=orange,
}

\theoremstyle{definition}
\newtheorem{dfn}{Definition}[section]
\newtheorem{thm}[dfn]{Theorem}
\newtheorem{conj}[dfn]{Conjecture}
\newtheorem{lem}[dfn]{Lemma}
\newtheorem{cor}[dfn]{Corollary}
\newtheorem*{rem}{Remark}

\let\oldenumerate\enumerate
\renewcommand{\enumerate}{
   \oldenumerate
   \setlength{\itemsep}{0cm}
   \setlength{\parskip}{0cm}
}
\let\olditemize\itemize
\renewcommand{\itemize}{
   \olditemize
   \setlength{\itemsep}{0cm}
   \setlength{\parskip}{0cm}
}
\let\olddescription\description
\renewcommand{\description}{
   \olddescription
   \setlength{\itemsep}{0cm}
   \setlength{\parskip}{0cm}
}

\title{2-distance 20-coloring of planar graphs with maximum degree 6}

\author{Kengo Aoki\thanks{Department of Informatics, Graduate School of Integrated Science and Technology, Shizuoka University, 3-5-1 Johoku, Chuo-ku, Hamamatsu-shi, 432-8011, Shizuoka, Japan, aoki.kengo.19@shizuoka.ac.jp}}
\date{}
\begin{document}
\maketitle
\begin{abstract}
A 2-distance $k$-coloring of a graph $G$ is a proper $k$-coloring such that any two vertices at distance two or less get different colors.
The 2-distance chromatic number of $G$ is the minimum $k$ such that $G$ has a 2-distance $k$-coloring, denoted by $\chi_2(G)$.
In this paper, we show that $\chi_2(G) \leq 20$ for every planar graph $G$ with maximum degree at most six, which improves a former bound $\chi_2(G) \leq 21$.
\end{abstract}

\section{Introduction}
All graphs considered in this paper are simple, finite, and planar.
For a graph $G$, we denote the set of vertices, the set of edges, and the set of faces by $V(G)$, $E(G)$, and $F(G)$, respectively.
For a graph $G$ and a vertex $v$ in $G$, let $G - v$ denote the graph obtained from $G$ by deleting the vertex $v$ and all edges incident with $v$.
The set of neighbours of a vertex $v$ in a graph $G$ is denoted by $N_G(v)$.
The \textit{degree} of a vertex $v$ in a graph $G$, denoted by $d_G(v)$, is the number of edges of $G$ incident with $v$.
The maximum degree and minimum degree of a graph $G$ are denoted by $\Delta(G)$ and $\delta(G)$ ($\Delta$ and $\delta$ for short).
A vertex of degree $k$ (respectively, at least $k$, at most $k$) is said to be a $k$-\textit{vertex} (respectively, $k^+$-vertex, $k^-$-vertex).
For $X \subseteq V(G)$, let $G[X]$ denote the subgraph of $G$ induced by $X$.
A face is said to be \textit{incident} with the vertices and edges in its boundary, and two faces are \textit{adjacent} if their boundaries have an edge in common.
The \textit{degree} of a face $f$ in a graph $G$, denoted by $d(f)$, is the number of edges in its boundary.
A face of degree $k$ (respectively, at least $k$, at most $k$) is said to be a $k$-\textit{face} (respectively, $k^+$-face, $k^-$-face).
A $[v_1v_2\ldots v_k]$ is a $k$-face with vertices $v_1,v_2, \ldots ,v_k$ on its boundary.
For a vertex $v$ in a graph $G$, let $m_k(v)$ denote the number of $k$-faces incident with $v$,  and let $n_k(v)$ denote the number of $k$-vertices adjacent to $v$.

Let $\phi$ be a partial coloring of a graph $G$.
For a vertex $v$ in a graph $G$, let $C_\phi(v)$ denote the set of colors assigned to the vertices within distance two from $v$.
A 2-distance $k$-coloring of a graph $G$ is a mapping $\phi : V(G) \rightarrow \{1,2,\ldots,k\}$ such that $\phi(v_1) \neq \phi(v_2)$ for any two vertices $v_1,v_2 \in V(G)$ with $d_G(v_1, v_2) \leq 2$, where $d_G(v_1, v_2)$ is the distance between the two vertices $v_1$ and $v_2$.
The 2-distance chromatic number of $G$ is the minimum $k$ such that $G$ has a 2-distance $k$-coloring, denoted by $\chi_2(G)$.
Let $d_2(v)$ denote the number of vertices within distance two from a vertex $v$.
Any definitions and notations not explicitly stated in this paper conform to those in \cite{MR2368647}.

The study of 2-distance coloring originated from the research on square coloring, which was first introduced by Kramer and Kramer \cite{kramer1969probleme, kramer1969farbungsproblem}.
The square of a graph $G$, denoted by $G^2$, is obtained by adding edges between all pairs of vertices that have a common neighbour in $G$.
In 1977, Wegner made the following conjecture.
\begin{conj}\cite{wegner1977graphs}\label{wegner}
If $G$ is a planar graph, then $\chi_2(G) \leq 7$ if $\Delta = 3$, $\chi_2(G) \leq \Delta+5$ if $4\leq \Delta \leq 7$, and $\chi_2(G) \leq \lfloor \frac{3\Delta}{2} \rfloor + 1$ if $\Delta \geq 8$.
\end{conj}
The case of $\Delta = 3$ was independently proven by Thomassen \cite{THOMASSEN2018192} and by Hartke et al. \cite{hartke2016chromatic}.
Havet et al. \cite{HAVET2007515, havet2017list} proved that the conjecture holds asymptotically.
Bousquet et al. \cite{MR4695743} proved that $\chi_2(G) \leq 12$ for $\Delta \leq 4$ using an automatic discharging method.
Deniz \cite{deniz20232distance} proved that $\chi_2(G) \leq 16$ for $\Delta \leq 5$.
For a comprehensive overview of 2-distance coloring and related research, we refer the reader to \cite{MR4578038}.

The upper bound on $\chi_2(G)$ for $\Delta = 6$ has been gradually improving.
Zhu and Bu \cite{zhu2018minimum} proved that $\chi_2(G) \leq 5\Delta - 7$ for $\Delta \geq 6$, which was improved by Krzyzi\'{n}ski et al. \cite{MR4740061} to $\chi_2(G) \leq 3\Delta + 4$ for $\Delta \geq 6$.
In this paper, we show that $\chi_2(G) \leq 20$ for every planar graph $G$ with $\Delta \leq 6$, which improves the result of 
$\chi_2(G) \leq 21$ proved by Bousquet et al. \cite{bousquet2023improved}.
We prove the following theorem.
\begin{thm}\label{main}
If $G$ is a planar graph with maximum degree $\Delta \leq 6$, then $\chi_2(G) \leq 20$.
\end{thm} 

\begin{rem}
Recently, Deniz \cite{deniz20242distance} posted the first version of the proof on arXiv on March 18, 2024, showing that $\chi_2(G) \leq 2\Delta + 7$ for planar graphs.
The proof in this version covers the cases $\Delta = 6$, $7$, and $8$.
According to this inequality, when $\Delta = 6$, it follows that $\chi_2(G) \leq 19$.
However, the proof for the case $\Delta = 6$ in this initial version is incomplete and requires further elaboration.
\end{rem}

\section{Reducible configurations}
Let $G$ be a minimum counterexample to Theorem~\ref{main} with minimum $|V(G)| + |E(G)|$.
That is $G$ is a planar graph with $\chi_2(G) > 20$, such that for any planar subgraph $G'$ with $\Delta(G') \leq \Delta(G)$ and $|V(G')| + |E(G')| < |V(G)| + |E(G)|$, we have  $\chi_2(G') \leq 20$.
Obviously, $G$ is a connected graph.

Let $C = \{1,2, \ldots , 20\}$ be a set of colors.
We call a graph $G'$ \textit{proper} with respect to $G$ if $G'$ is obtained from $G$ by deleting some edges or vertices and adding some edges such that for any two vertices $v_1$, $v_2 \in V(G) \cap V(G')$ with $d_G(v_1,v_2) \leq 2$, we have $d_{G'}(v_1,v_2) \leq 2$.
This definition of proper is the same as the one used in \cite{deniz20232distance, hou2023coloring}.
In this section, we present some reducible configurations of $G$.
The proofs of the lemmas generally follow a similar pattern:
We construct a graph $G'$ that is proper with respect to $G$ by deleting a vertex $v$ from $G$ and adding some edges.
By the minimality of $G$, there exists a 2-distance 20-coloring $\phi'$ of $G'$.
Let $\phi$ be a coloring of $G$ such that every vertex in $V(G)$, except for the deleted vertex $v$, is colored using $\phi'$.
If $|C| - |C_{\phi}(v)| \geq 1$, then a safe color exists for $v$.
By coloring $v$ with the safe color, $\phi'$ can be extended to a 2-distance 20-coloring $\phi$ of $G$.
This implies that $\chi_2(G) \leq 20$, which is a contradiction.
The essence of the proof is to construct a proper $G'$ such that $|C_{\phi}(v)| \leq d_2(v) \leq 19$.

\begin{lem}\label{lem1}%Lemma1
We have $\delta(G) \geq 3$. 
\end{lem}
\begin{proof}
Assume that $G$ contains a 1-vertex $v$.
It is clear that $G' = G - v$ is proper with respect to $G$.
By the minimality of $G$, $G'$ has a 2-distance 20-coloring $\phi'$.
Let $\phi$ be a coloring of $G$ such that every vertex in $V(G)$, except for $v$, is colored using $\phi'$.
Since $\Delta \leq 6$, it follows that $|C_\phi(v)| \leq 6$ and $|C| - |C_\phi(v)| \geq 14$.
Therefore, there exists a safe color for $v$.
%If $v$ is colored with a safe color, then $\phi$ becomes a 2-distance 20-coloring of $G$, which is a contradiction.
By coloring $v$ with the safe color, $\phi$ becomes a 2-distance 20-coloring of $G$, a contradiction.
Next, we assume that $G$ has a 2-vertex $v$ with $N_G(v) = \{x, y\}$.
Let $G' = G - v + \{xy\}$.
The graph $G'$ is proper with respect to $G$.
By the minimality of $G$, $G'$ has a 2-distance 20-coloring $\phi'$.
Let $\phi$ be a coloring of $G$ such that every vertex in $V(G)$, except for $v$, is colored using $\phi'$.
Since $\Delta \leq 6$, it follows that $|C_\phi(v)| \leq 12$ and $|C| - |C_\phi(v)| \geq 8$.
Therefore, we can color $v$ with a safe color, a contradiction.
\end{proof}

\begin{lem}\label{lem2}%Lemma2
Let $v$ be a 3-vertex.
Then, 
	\begin{enumerate}
	\item[(1)] $n_{5^-}(v) = 0$,
	\item[(2)] $m_3(v) = 0$, and
	\item[(3)] $m_4(v) \leq 1$.
	\end{enumerate}
\end{lem}
\begin{proof}
Let $v_1$, $v_2$, and $v_3$ be the neighbours of $v$.
(1) Assume that $v$ is adjacent to a $5^-$-vertex.
Without loss of generality, let $v_1$ be a $5^-$-vertex.
Let $G' = G - v + \{v_1v_2, v_1v_3\}$.
The graph $G'$ is proper with respect to $G$.
By the minimality of $G$, $G'$ has a 2-distance 20-coloring $\phi'$.
Let $\phi$ be a coloring of $G$ such that every vertex in $V(G)$, except for $v$, is colored using $\phi'$.
Since $\Delta \leq 6$, it follows that $|C_\phi(v)| \leq 17$ and $|C| - |C_\phi(v)| \geq 3$.
Therefore, there exists a safe color for $v$.
By coloring $v$ with the safe color, $\phi$ becomes a 2-distance 20-coloring of $G$, a contradiction.

(2) Assume that $v$ is incident to a 3-face $[vv_1v_2]$.
Let $G' = G - v + \{v_1v_3\}$.
The graph $G'$ is proper with respect to $G$.
By the minimality of $G$, $G'$ has a 2-distance 20-coloring $\phi'$.
Let $\phi$ be a coloring of $G$ such that every vertex in $V(G)$, except for $v$, is colored using $\phi'$.
Since $|C_\phi(v)| \leq 16$, we can color $v$ with a safe color, a contradiction.

(3) Assume that $v$ is incident to two 4-faces $[vv_1xv_2]$ and $[vv_2yv_3]$.
It is clear that $G' = G - v + \{v_1v_3\}$ is proper with respect to $G$.
By the minimality of $G$, $G'$ has a 2-distance 20-coloring $\phi'$.
Let $\phi$ be a coloring of $G$ such that every vertex in $V(G)$, except for $v$, is colored using $\phi'$.
Since $|C_\phi(v)| \leq 16$, we can color $v$ with a safe color, a contradiction.
\end{proof}

\begin{lem}\label{lem3}%Lemma3
Let $v$ be a 4-vertex. 
Then $m_3(v) \leq 2$.
In particular, if $m_3(v) = 2$, then $m_4(v) = 0$, $n_6(v) = 4$, and $m_3(w) \leq 4$ for any 6-vertex $w$ adjacent to $v$.
\begin{proof}
Let $v_1$, $v_2$, $v_3$, and $v_4$ be the neighbours of $v$ in clockwise order.
First, we show that $m_3(v) \leq 2$.
Assume that $v$ is incident to three 3-faces $[vv_1v_2]$, $[vv_2v_3]$, and $[vv_3v_4]$.
Let $G' = G - v + \{v_1v_4\}$.
The graph $G'$ is proper with respect to $G$.
By the minimality of $G$, $G'$ has a 2-distance 20-coloring $\phi'$.
Let $\phi$ be a coloring of $G$ such that every vertex in $V(G)$, except for $v$, is colored using $\phi'$.
Since $\Delta \leq 6$, it follows that $|C_\phi(v)| \leq 18$ and $|C| - |C_\phi(v)| \geq 2$.
Therefore, there exists a safe color for $v$.
By coloring $v$ with the safe color, $\phi$ becomes a 2-distance 20-coloring of $G$, a contradiction.

Now, we consider the case $m_3(v) = 2$.	
Let $f_1$ and $f_2$ be two 3-faces incident to $v$.
First, we show that $m_4(v) = 0$.
Assume that $v$ is incident to a 4-face $[vv_1xv_2]$.
If $f_1$ and $f_2$ are adjacent, say $f_1 = [vv_2v_3]$ and $f_2 = [vv_3v_4]$, then let $G' = G - v + \{v_1v_4\}$.
If $f_1$ and $f_2$ are not adjacent, say $f_1 = [vv_2v_3]$ and $f_2 = [vv_4v_1]$, then let $G' = G - v + \{v_3v_4\}$.
In both cases, $G'$ is proper with respect to $G$ and by the minimality of $G$, $G'$ has a 2-distance 20-coloring $\phi'$.
Let $\phi$ be a coloring of $G$ such that every vertex in $V(G)$, except for $v$, is colored using $\phi'$.
Since $|C_\phi(v)| \leq 19$ in each case, we can color $v$ with a safe color, a contradiction.

Next, we prove that $n_6(v) = 4$.
Assume that $v$ is incident to a $5^-$-vertex.
Without loss of generality, let $v_1$ be a $5^-$-vertex.
If $f_1$ and $f_2$ are adjacent, say $f_1 = [vv_1v_2]$ and $f_2 = [vv_2v_3]$, then let $G' = G - v + \{v_2v_4\}$.
If $f_1$ and $f_2$ are not adjacent, say $f_1 = [vv_1v_2]$ and $f_2 = [vv_3v_4]$, then let $G' = G - v + \{v_2v_3, v_4v_1\}$.
In both cases, $G'$ is proper with respect to $G$ and $d_2(v) \leq 19$, a contradiction.

Finally, we prove that $m_3(w) \leq 4$ for any 6-vertex $w$ adjacent to $v$.
To show that $w$ cannot be incident to more than five 3-faces, it suffices to prove that no edge in $G[N_G(v)]$ is contained in two 3-faces of $G$.
Let $f_1 = [vv_1v_2]$.
Assume that the edge $v_1v_2$ is contained in two 3-faces of $G$.
This implies that there exists a vertex $x$ such that $x$ is a common neighbour of $v_1$ and $v_2$.
If $f_1$ and $f_2$ are adjacent, say $f_2 = [vv_2v_3]$, then let $G' = G - v + \{v_2v_4\}$.
If $f_1$ and $f_2$ are not adjacent, say $f_2 = [vv_3v_4]$, then let $G' = G - v + \{v_2v_3, v_4v_1\}$.
In both cases, $G'$ is proper with respect to $G$ and $d_2(v) \leq 19$, a contradiction.
\end{proof}
\end{lem}

\begin{lem}\label{lem4}%Lemma4
Let $v$ be a 4-vertex with $m_3(v) = 1$. 
Then $m_4(v) \leq 2$.
In particular, if $1 \leq m_4(v) \leq 2$, then $n_4(v) = 0$, and $n_5(v) \leq 1$.
\end{lem}
\begin{proof}
Let $v_1$, $v_2$, $v_3$, and $v_4$ be the neighbours of $v$ in clockwise order and let $[vv_1v_2]$ be a 3-face incident to $v$.
First, we show that $m_4(v) \leq 2$.
Assume that $v$ is incident to three 4-faces $[vv_2xv_3]$, $[vv_3yv_4]$, and $[vv_4zv_1]$.
Let $G' = G - v + \{v_2v_3, v_1v_4\}$.
The graph $G'$ is proper with respect to $G$.
By the minimality of $G$, $G'$ has a 2-distance 20-coloring $\phi'$.
Let $\phi$ be a coloring of $G$ such that every vertex in $V(G)$, except for $v$, is colored using $\phi'$.
Since $\Delta \leq 6$, it follows that $|C_\phi(v)| \leq 19$ and $|C| - |C_\phi(v)| \geq 1$.
Therefore, there exists a safe color for $v$.
By coloring $v$ with the safe color, $\phi$ becomes a 2-distance 20-coloring of $G$, a contradiction.

Now, we suppose that $m_4(v) = 2$.
Let $f_1$ and $f_2$ be two 4-faces incident to $v$.
First, we prove that $n_4(v) = 0$.
Assume that $v$ is adjacent to a 4-vertex.
Let $v_1$ be a 4-vertex.
Regardless of whether $f_1$ and $f_2$ are adjacent, let $G' = G - v + \{v_1v_3, v_1v_4\}$.
(If a vertex $v_i \in N_G(v)$ other than $v_1$ is a 4-vertex, then we construct $G'$ by deleting $v$ and adding edges from $v_i$ to each neighbour $v_j$ of $v$ with $v_iv_j \notin E(G)$.)
The graph $G'$ is proper with respect to $G$.
By the minimality of $G$, $G'$ has a 2-distance 20 coloring $\phi'$.
Let $\phi$ be a coloring of $G$ such that every vertex in $V(G)$, except for $v$, is colored using $\phi'$.
Since $|C_\phi(v)| \leq 18$, there exists a safe color for $v$, a contradiction.
Second, we show that $n_5(v) \leq 1$.
Assume that $v$ is adjacent to two 5-vertices.
There are six possible arrangements of two 5-vertices among the four neighbors of $v$.
Regardless of whether $f_1$ and $f_2$ are adjacent, the construction of $G'$ depends on which neighbours of $v$ are the two 5-vertices.
If $v_1$ is one of the two 5-vertices, then we construct $G' = G - v + \{v_1v_3, v_1v_4\}$.
Similarly, if $v_2$ is one of the two 5-vertices, then we construct $G' = G - v + \{v_2v_3, v_2v_4\}$.
Otherwise, if $v_3$ and $v_4$ are two 5-vertices, then we construct $G' = G - v + \{v_2v_3, v_3v_4, v_4v_1\}$.  
In all cases, $G'$ is proper with respect to $G$.
By the minimality of $G$, $G'$ has a 2-distance 20-coloring $\phi'$.
Let $\phi$ be a coloring of $G$ such that every vertex in $V(G)$, except for $v$, is colored using $\phi'$.
Since $|C_\phi(v)| \leq 18$, we can color $v$ with a safe color, a contradiction.

Next, we suppose that $m_4(v) =1$.
Let $f_1$ be a 4-face incident to $v$.
The proof of $n_4(v) = 0$ is similar to the proof when we supposed that $m_4(v) = 2$.
Assume that $v$ is adjacent to a 4-vertex and let $v_1$ be a 4-vertex.
We construct $G'$ in the same way as before, regardless of the position of $f_1$: $G' = G - v + \{v_1v_3, v_1v_4\}$.
The graph $G'$ is proper with respect to $G$.
By the minimality of $G$, $G'$ has a 2-distance 20-coloring $\phi'$.
Let $\phi$ be a coloring of $G$ such that every vertex in $V(G)$, except for $v$, is colored using $\phi'$.
The only difference is that, in this case, $|C_\phi(v)| \leq 19$.
We can color $v$ with a safe color, a contradiction.
Finally, we prove that $n_5(v) \leq 1$.
Assume that $v$ is adjacent to two 5-vertices.
There are six possible arrangements of two 5-vertices among the four neighbors of $v$.
We consider two cases based on the position of $f_1$.
Case 1: $f_1 = [vv_2xv_3]$.
If $v_1$ is one of the two 5-vertices, then we construct $G' = G - v + \{v_1v_3, v_1v_4\}$.
Similarly, if $v_2$ is one of the two 5-vertices, then we construct $G' = G - v + \{v_2v_3, v_2v_4\}$.
Otherwise, if $v_3$ and $v_4$ are the two 5-vertices, then we construct $G' = G - v + \{v_2v_3, v_3v_4, v_4v_1\}$.
Case 2: $f_1 = [vv_3xv_4]$.
In this case, we construct $G' = G - v + \{v_2v_3, v_4v_1\}$, regardless of which neighbours of $v$ are the two 5-vertices.
In all cases, $G'$ is proper with respect to $G$.
Since $d_2(v) \leq 19$, there exists a safe color for $v$, a contradiction.
\end{proof}

Now, we discuss the properties of a 5-vertex in $G$.
Let $v$ be a 5-vertex and let $v_1,v_2, \ldots, v_5$ be the neighbours of $v$ in clockwise order.

\begin{lem}\label{lem5}%Lemma5
Let $v$ be a 5-vertex. 
If $m_3(v) = 5$, then $n_6(v) = 5$ and $m_3(w) \leq 4$ for any 6-vertex $w$ adjacent to $v$.
\begin{proof}
Suppose that $v$ is incident to five 3-faces.
First, we show that $n_6(v) = 5$.
Assume that $v$ is adjacent to a $5^-$-vertex.
Without loss of generality, let $v_1$ be a $5^-$-vertex.
It is clear that $G' = G - v$ is proper with respect to $G$.
By the minimality of $G$, $G'$ has a 2-distance 20-coloring $\phi'$.
Let $\phi$ be a coloring of $G$ such that every vertex in $V(G)$, except for $v$, is colored using $\phi'$.
Since $\Delta \leq 6$, it follows that $|C_\phi(v)| \leq 19$ and $|C| - |C_\phi(v)| \geq 1$.
Therefore, there exists a safe color for $v$.
By coloring $v$ with the safe color, $\phi$ becomes a 2-distance 20-coloring of $G$, a contradiction.

Next, we prove that $m_3(w) \leq 4$ for any 6-vertex $w$ adjacent to $v$.
To show that $w$ cannot be incident to more than five 3-faces, it suffices to prove that no edge in $G[N_G(v)]$ is contained in two 3-faces of $G$.
Without loss of generality, we assume that the edge $v_1v_2$ is contained in two 3-faces.
This implies that there exists a vertex $x$ such that $x$ is a common neighbour of $v_1$ and $v_2$.
It is clear that $G' = G - v$ is proper with respect to $G$.
Since $d_2(v) \leq 19$, there exists a safe color for $v$, a contradiction.
\end{proof}
\end{lem}

\begin{lem}\label{lem6}%Lemma6
Let $v$ be a 5-vertex with $m_3(v) = 4$ and $m_4(v)$ = 1.
Then $n_{4^-}(v) = 0$, $n_5(v) \leq 1$, and $m_3(w) \leq 5$ for any 6-vertex $w$ adjacent to $v$. 
\begin{proof}
Suppose that $v$ is incident to four 3-faces $[vv_1v_2]$, $[vv_2v_3]$, $[vv_3v_4]$, $[vv_4v_5]$, and one 4-face $[vv_5xv_1]$.
First, we show that $n_{4^-}(v) = 0$.
By Lemma~\ref{lem2}(1), $v$ is not adjacent to any 3-vertex.
Thus it suffices to show that $v$ is not adjacent to any 4-vertex.
Assume that $v$ is adjacent to a 4-vertex.
Let $G' = G - v + \{v_5v_1\}$.
The graph $G'$ is proper with respect to $G$.
By the minimality of $G$, $G'$ has a 2-distance 20-coloring $\phi'$.
Let $\phi$ be a coloring of $G$ such that every vertex in $V(G)$, except for $v$, is colored using $\phi'$.
Since $\Delta \leq 6$, it follows that $|C_{\phi}(v)| \leq 19$ and $|C| - |C_\phi(v)| \geq 1$.
Therefore, there exists a safe color for $v$.
By coloring $v$ with the safe color, $\phi$ becomes a 2-distance 20-coloring of $G$, a contradiction.

Next, we assume that $v$ is incident to two 5-vertices.
It is clear that $G' = G - v + \{v_5v_1\}$ is proper with respect to $G$.
By the minimality of $G$, $G'$ has a 2-distance 20-coloring $\phi'$.
Let $\phi$ be a coloring of $G$ such that every vertex in $V(G)$, except for $v$, is colored using $\phi'$.
Since $|C_{\phi}(v)| \leq 19$, we can color $v$ with a safe color, a contradiction.
Thus $n_5(v) \leq 1$ holds.
This implies that $n_6(v) \geq 4$.

Finally, we prove that $m_3(w) \leq 5$ for any 6-vertex $w$ adjacent to $v$.
Assume that $v_1$ is a 6-vertex.
Since $v_1$ is incident to a 4-face, $v_1$ can be incident to at most five 3-faces.
By symmetry, the same holds if we assume that $v_5$ is a 6-vertex.
Next, we prove that if $v_2$, $v_3$, or $v_4$ is a 6-vertex, then it can be incident to at most five 3-faces.
Without loss of generality, we assume that $v_2$ is a 6-vertex and is incident to six 3-faces.
In this case, each of the edges $v_1v_2$ and $v_2v_3$ is contained in two 3-faces.
Let $G' = G - v + \{v_1v_5\}$.
The graph $G'$ is proper with respect to $G$.
Since $d_2(v) \leq 19$, there exists a safe color for $v$, a contradiction.
\end{proof}
\end{lem}

\begin{lem}\label{lem7}%Lemma7
Let $v$ be a 5-vertex with $m_3(v) = 4$ and $m_{5^+}(v)$ = 1.
Then $n_3(v) = 0$, $n_4(v) \leq 1$, and $n_5(v) \leq 2$.
In particular, if $n_4(v) = 1$, then $n_5(v) = 0$.
\begin{proof}
Suppose that $v$ is incident to four 3-faces $[vv_1v_2]$, $[vv_2v_3]$, $[vv_3v_4]$, $[vv_4v_5]$, and one $5^+$-face that contains $v_1$ and $v_5$. 
Obviously, we have $n_3(v) = 0$ by Lemma~\ref{lem2}(1).
Now, we show that $n_4(v) \leq 1$.
Assume that $v$ is adjacent to two 4-vertices.
Let $G' = G - v + \{v_1v_5\}$.
The graph $G'$ is proper with respect to $G$.
By the minimality of $G$, $G'$ has a 2-distance 20-coloring $\phi'$.
Let $\phi$ be a coloring of $G$ such that every vertex in $V(G)$, except for $v$, is colored using $\phi'$.
Since $\Delta \leq 6$, it follows that $|C_{\phi}(v)| \leq 18$ and $|C| - |C_\phi(v)| \geq 2$.
Therefore, there exists a safe color for $v$.
By coloring $v$ with the safe color, $\phi$ becomes a 2-distance 20-coloring of $G$, a contradiction.

Next, we prove that $n_5(v) \leq 2$.
Assume that $v$ is adjacent to three 5-vertices.
It is clear that $G' = G - v + \{v_1v_5\}$ is proper with respect to $G$.
By the minimality of $G$, $G'$ has a 2-distance 20-coloring $\phi'$.
Let $\phi$ be a coloring of $G$ such that every vertex in $V(G)$, except for $v$, is colored using $\phi'$.
Since $|C_{\phi}(v)| \leq 19$, we can color $v$ with a safe color, a contradiction.

Finally, we consider the case $n_4(v) = 1$.
Assume that $v$ is adjacent to a 5-vertex.
Obviously, $G' = G - v + \{v_1v_5\}$ is proper with respect to $G$.
Since $d_2(v) \leq 19$, there exists a safe color for $v$, a contradiction.
\end{proof}
\end{lem}

\begin{lem}\label{lem8}%Lemma8
Let $v$ be a 5-vertex with $m_3(v) = 4$ and $m_{5^+}(v)$ = 1.
Then the number of 4-vertices, 5-vertices, and 6-vertices adjacent to $v$ must be one of the following:
\begin{enumerate}
	\item[(a)] ($n_4(v)$, $n_5(v)$, $n_6(v)$) = (1, 0, 4), 
	\item[(b)] ($n_4(v)$, $n_5(v)$, $n_6(v)$) = (0, 2, 3), 
	\item[(c)] ($n_4(v)$, $n_5(v)$, $n_6(v)$) = (0, 1, 4), or
	\item[(d)] ($n_4(v)$, $n_5(v)$, $n_6(v)$) = (0, 0, 5).
	\end{enumerate}
	
Moreover, let $w$ be any 6-vertex adjacent to $v$. Then the following hold:
\begin{enumerate}
	\item[(1)] If $v$ is in case (a), then $m_3(w) \leq 4$.
	\item[(2)] If $v$ is in case (b), then $m_3(w) \leq 4$.
	\item[(3)] If $v$ is in case (c), then there exists at least one 6-vertex $w$ with $m_3(w) \leq 5$. 
	\item[(4)] If $v$ is in case (d), then there exist at least two 6-vertices $w_1$, $w_2$ with $m_3(w_1) \leq 5$ and $m_3(w_2) \leq 5$.
	\end{enumerate}
\end{lem}
\begin{proof}
The first statement of the lemma follows directly from Lemma~\ref{lem7}.
We now prove the remaining statements, from (1) to (4).
Suppose that $v$ is incident to four 3-faces $[vv_1v_2]$, $[vv_2v_3]$, $[vv_3v_4]$, $[vv_4v_5]$, and one $5^+$-face that contains $v_1$ and $v_5$. 

(1) To show that $w$ cannot be incident to more than five 3-faces, it suffices to prove that no edge in $G[N_G(v)]$ is contained in two 3-faces of $G$.
Assume that the edge $v_1v_2$ is contained in two 3-faces.
It is clear that $G' = G - v + \{v_5v_1\}$ is proper with respect to $G$.
By the minimality of $G$, $G'$ has a 2-distance 20-coloring $\phi'$.
Let $\phi$ be a coloring of $G$ such that every vertex in $V(G)$, except for $v$, is colored using $\phi'$.
Since $n_4(v) = 1$ and $n_6(v) = 4$, it follows that $|C_{\phi}(v)| \leq 19$ and $|C| - |C_\phi(v)| \geq 1$.
Therefore, there exists a safe color for $v$.
By coloring $v$ with the safe color, $\phi$ becomes a 2-distance 20-coloring of $G$, a contradiction.

(2) The proof is similar to that of (1). To show that $w$ cannot be incident to more than five 3-faces, it suffices to prove that no edge in $G[N_G(v)]$ is contained in two 3-faces of $G$. Assume that the edge $v_1v_2$ is contained in two 3-faces. 
As in (1), let $G' = G - v + \{v_5v_1\}$, which is proper with respect to $G$. 
By the minimality of $G$, $G'$ has a 2-distance 20-coloring $\phi'$.
 Let $\phi$ be a coloring of $G$ such that every vertex in $V(G)$, except for $v$, is colored using $\phi'$.
In case (b), we have $n_5(v) = 2$ and $n_6(v) = 3$, which leads to $|C_{\phi}(v)| \leq 19$.
Therefore, there exists a safe color for $v$. By coloring $v$ with the safe color, $\phi$ becomes a 2-distance 20-coloring of $G$, a contradiction.

(3) In case (c), we have $n_5(v) = 1$ and $n_6(v) = 4$.
It follows that at least one of $v_1$ and $v_5$ must be a 6-vertex.
Let $w$ be such a 6-vertex.
Since $w$ is incident to one $5^+$-face, it can be incident to at most five 3-faces.
Therefore, (3) holds.

(4) In case (d), all neighbours of $v$ are 6-vertices.
It follows that both $v_1$ and $v_5$ are 6-vertices.
Let $w_1 = v_1$ and $w_2 = v_5$.
Since each of $w_1$ and $w_2$ is incident to one $5^+$-face, it can be incident to at most five 3-faces.
Therefore, (4) holds.
\end{proof}

Next, we examine the properties of a 6-vertex in $G$.
Let $v$ be a 6-vertex and let $v_1,v_2, \ldots, v_6$ be the neighbours of $v$ in clockwise order.

\begin{figure}[ht]
  \begin{minipage}[b]{0.45\linewidth}
    \centering
    \includegraphics[scale=0.24]{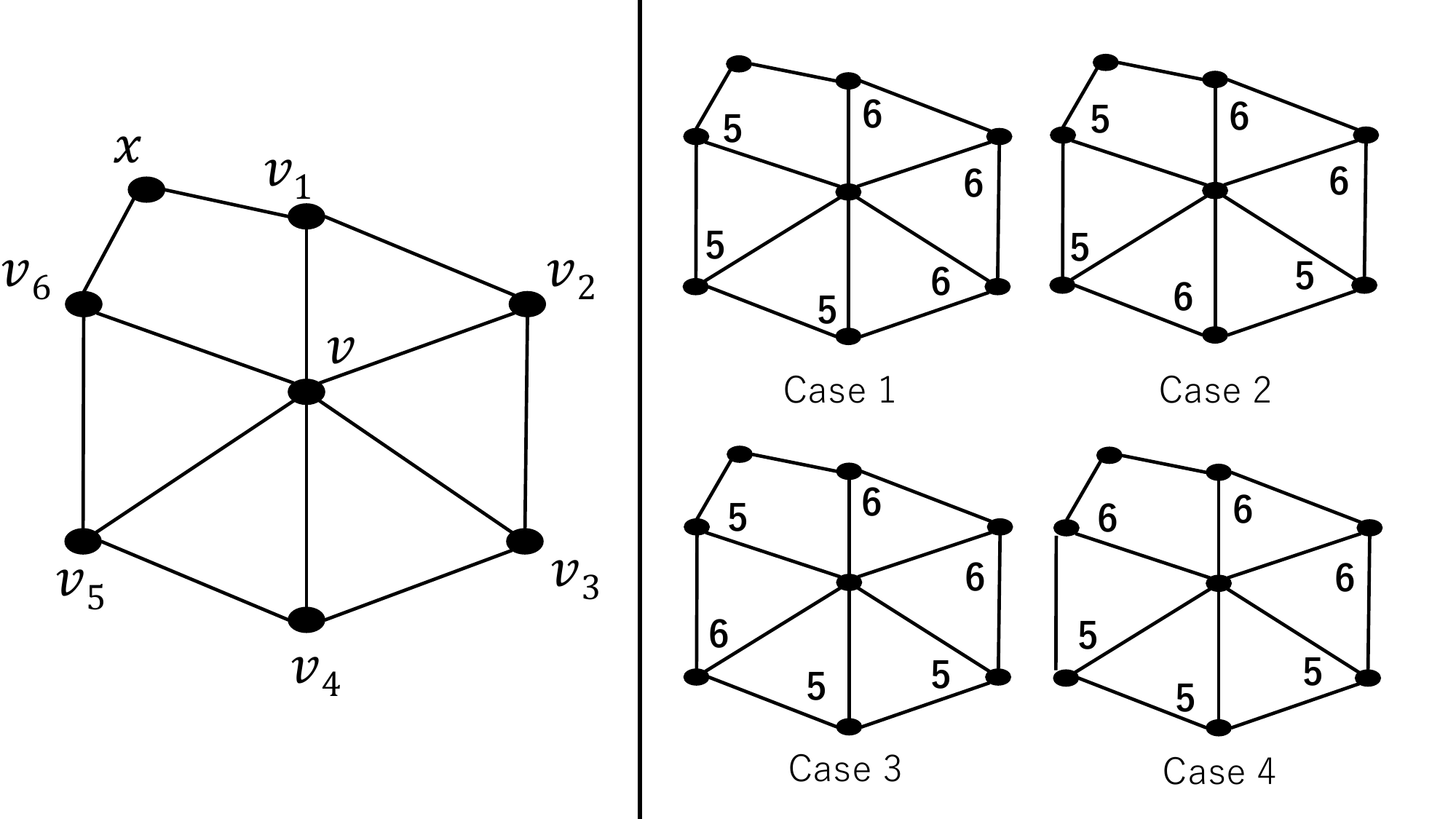}
  \end{minipage}
  \begin{minipage}[b]{0.45\linewidth}
    \centering
    \includegraphics[scale=0.24]{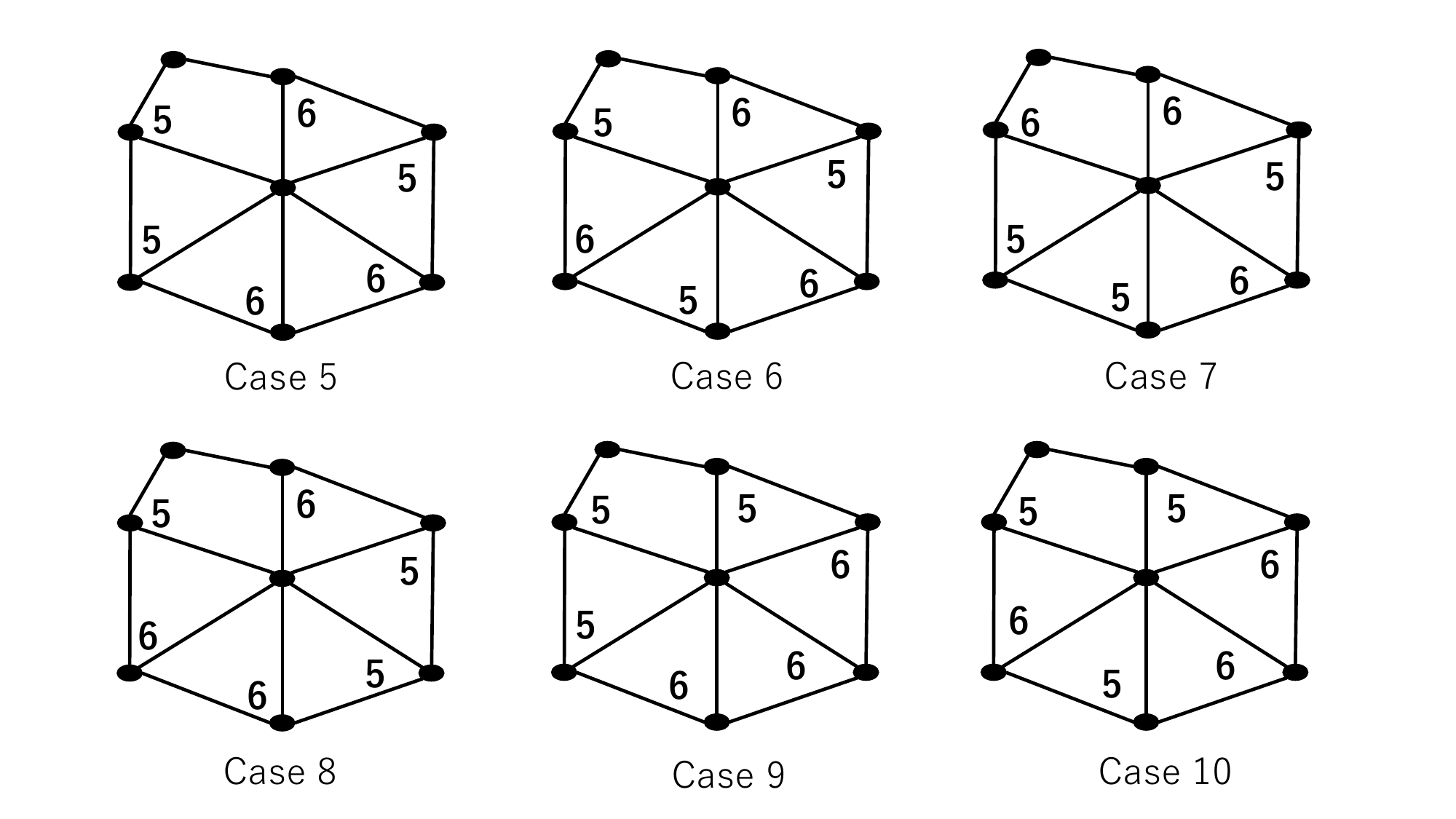}
  \end{minipage}
  \caption{Illustrations of Lemma~\ref{lem9}(4.3). }
  \label{fig1}
\end{figure}

\begin{lem}\label{lem9}%Lemma9
Let $v$ be a 6-vertex with $m_3(v) = 5$ and $m_4(v) = 1$ and let $u$ be any 5-vertex adjacent to $v$.
Then the following hold:
	\begin{enumerate}
	\item[(1)] $n_4(v) \leq 2$.
	\item[(2)] If $n_4(v) = 2$, then $n_5(v) \leq 1$.
	\item[(3)] If $n_4(v) = 1$, then $n_5(v) \leq 3$. 
	Moreover, if $n_5(v) = 3$, then $m_3(u) \leq 3$.
	\item[(4)] If $n_4(v) = 0$, then $n_5(v) \leq 5$.
	Moreover, the following hold:
		\begin{enumerate}
		\item[(4.1)] If $n_5(v) = 5$, then $m_3(u) \leq 3$.
		\item[(4.2)] If $n_5(v) = 4$, then there exist at most two 5-vertices $u$ with $m_3(u) \geq 4$.
		\item[(4.3)] If $n_5(v) = 3$, then there exist at most two 5-vertices $u$ with $m_3(u) \geq 4$.
		\end{enumerate}
	\end{enumerate}
\end{lem}
\begin{proof}
Suppose that $v$ is incident to five 3-faces $[vv_1v_2]$, $[vv_2v_3]$, $[vv_3v_4]$, $[vv_4v_5]$, $[vv_5v_6]$, and one 4-face $[vv_6xx_1]$.
By Lemma~\ref{lem2}(2), $v$ is not adjacent to any 3-vertex.

(1) In the proof of Lemma~\ref{lem3}, we showed that if $v$ is a 4-vertex with $m_3(v) = 2$, then no edge in $G[N_G(v)]$ is contained in two 3-faces.
Thus $v_2$, $v_3$, $v_4$, and $v_5$ cannot be 4-vertices.
Among the neighbours of $v$, at most two vertices, namely, $v_1$ and $v_6$, can be 4-vertices.

(2) We suppose that $v_1$ and $v_6$ are 4-vertices.
Assume that $v$ is adjacent to two 5-vertices.
Let $G' = G - v + \{v_1v_3, v_1v_5, v_1v_6\}$.
The graph $G'$ is proper with respect to $G$.
By the minimality of $G$, $G'$ has a 2-distance 20-coloring $\phi'$.
Let $\phi$ be a coloring of $G$ such that every vertex in $V(G)$, except for $v$, is colored using $\phi'$.
Since $\Delta \leq 6$, it follows that $|C_{\phi}(v)| \leq 19$ and $|C| - |C_\phi(v)| \geq 1$.
Therefore, there exists a safe color for $v$.
By coloring $v$ with the safe color, $\phi$ becomes a 2-distance 20-coloring of $G$, a contradiction.

(3) Without loss of generality, we suppose that $v_1$ is a 4-vertex.
Assume that $v$ is adjacent to four 5-vertices.
The graph $G' = G - v + \{v_1v_3, v_1v_5, v_1v_6\}$ is proper with respect to $G$.
By the minimality of $G$, $G'$ has a 2-distance 20-coloring $\phi'$.
Let $\phi$ be a coloring of $G$ such that every vertex in $V(G)$, except for $v$, is colored using $\phi'$.
Since $|C_{\phi}(v)| \leq 19$, we can color $v$ with a safe color, a contradiction.

Now, we consider the case $n_5(v)$ = 3.
To show that any 5-vertex $u$ adjacent to $v$ cannot be incident to more than four 3-faces, it suffices to prove that no edge in $G[N_G(v)]$ is contained in two 3-faces of $G$.
Assume that the edge $v_1v_2$ is contained in two 3-faces.
Let $G' = G - v + \{v_1v_3, v_1v_5, v_1v_6\}$.
The graph $G'$ is proper with respect to $G$.
Since $d_2(v) \leq 19$, there exists a safe color for $v$, a contradiction.

(4) We suppose that $v$ is not adjacent to any 4-vertex.
Assume that all neighbours of $v$ are 5-vertices.
Let $G' = G - v +  \{v_6v_1, v_3v_1, v_3v_5\}$.
The graph $G'$ is proper with respect to $G$.
By the minimality of $G$, $G'$ has a 2-distance 20-coloring $\phi'$.
Since $d_2(v) \leq 19$, there exists a safe color for $v$, a contradiction.

(4.1) Suppose that $v$ is adjacent to five 5-vertices.
To show that any 5-vertex $u$ adjacent to $v$ cannot be incident to more than four 3-faces, it suffices to prove that no edge in $G[N_G(v)]$ is contained in two 3-faces of $G$.
Assume that the edge $v_1v_2$ is contained in two 3-faces.
There are six possibilities for which neighbours of $v$ is a 6-vertex.
Due to symmetry, it suffices to consider the cases where $v_1$, $v_2$, or $v_3$ is a 6-vertex.
In each of these cases, let $G' = G - v + \{v_6v_1, v_4v_2, v_4v_6\}$.
Then $G'$ is proper with respect to $G$.
Since $d_2(v) \leq 19$, there exists a safe color for $v$, a contradiction.

(4.2) Suppose that $v$ is adjacent to four 5-vertices.
We consider two cases based on whether both $v_1$ and $v_6$ are 6-vertices or not.
First, we show that if $v_1$ and $v_6$ are both 6-vertices, then $m_3(v_3) \leq 3$ and $m_3(v_4) \leq 3$.
By symmetry, we only need to consider $v_3$.
Let $v_7$ and $v_8$ be the neighbours of $v_3$ other than $v$, $v_2$, and $v_4$.
Assume that $v_3$ is incident to four 3-faces.
We have two cases:
Case~1: The four 3-faces are $[vv_2v_3]$, $[vv_3v_4]$, $[v_2v_7v_3]$, and $[v_3v_8v_4]$. 
Case~2: The four 3-faces are $[vv_2v_3]$, $[vv_3v_4]$, $[v_2v_7v_3]$, and $[v_3v_7v_8]$.
In Case~1, let $G' = G - v_3 + \{v_7v_8\}$.
In Case~2, let $G' = G - v_3 + \{v_8v_4\}$.
In each case, $G'$ is proper with respect to $G$.
By the minimality of $G$, $G'$ has a 2-distance 20-coloring $\phi'$.
Let $\phi$ be a coloring of $G$ such that every vertex in $V(G)$, except for $v_3$, is colored using $\phi'$.
Since $|C_{\phi}(v_3)| \leq 18$, we can color $v_3$ with a safe color, a contradiction.
Therefore, if $v_1$ and $v_6$ are 6-vertices, then there are at most two 5-vertices $u$ adjacent to $v$ with $m_3(u) \geq 4$, namely $v_2$ and $v_5$.

Next, we discuss the case where $v_1$ is not a 6-vertex or $v_6$ is not a 6-vertex. 
To show that there exist at most two 5-vertices $u$ adjacent to $v$ with $m_3(u) \geq 4$, it suffices to prove that at most one edge in $G[N_G(v)]$ is contained in two 3-faces of $G$.
Assume that two edges $v_iv_{i+1}$ and $v_jv_{j+1}$ for $i, j \in \{1,2,3,4,5\}$ with $i \neq j$ are contained in two 3-faces. 
If $v_1$ and $v_2$ are 6-vertices, then we construct $G' = G - v + \{v_2v_4, v_4v_6, v_6v_1\}$.
(Otherwise, we construct $G'$ as follows: remove $v$, add the edge $v_6v_1$, and choose one 5-vertex $v_i$ in the neighbourhood of $v$ other than $v_1$ and $v_6$, and connect $v_i$ to the two vertices in the neighbourhood of $v$ that are at distance two from $v_i$.)
The graph $G'$ is proper with respect to $G$.
Since $d_2(v) \leq 19$, there exists a safe color for $v$, a contradiction.

(4.3)
Suppose that $v$ is adjacent to three 5-vertices.
There are twenty possible combinations of three 5-vertices.
However, by symmetry, we only discuss ten cases: (see Figure~\ref{fig1}.)
Case~1: The three 5-vertices are $v_4$, $v_5$, and $v_6$.
Case~2: The three 5-vertices are $v_3$, $v_5$, and $v_6$.
Case~3: The three 5-vertices are $v_3$, $v_4$, and $v_6$.
Case~4: The three 5-vertices are $v_3$, $v_4$, and $v_5$.
Case~5: The three 5-vertices are $v_2$, $v_5$, and $v_6$.
Case~6: The three 5-vertices are $v_2$, $v_4$, and $v_6$.
Case~7: The three 5-vertices are $v_2$, $v_4$, and $v_5$.
Case~8: The three 5-vertices are $v_2$, $v_3$, and $v_6$.
Case~9: The three 5-vertices are $v_1$, $v_5$, and $v_6$.
Case~10: The three 5-vertices are $v_1$, $v_4$, and $v_6$.

First, we consider Cases~3, 4, and 7.
We show that $m_3(v_4) \leq 3$ in these cases.
Assume that $v_4$ is incident to four 3-faces.
Let $v_7$ and $v_8$ be the neighbours of $v_4$ other than $v$, $v_3$, and $v_5$.
Since $v_4$ is already incident to two 3-faces, namely $[vv_3v_4]$ and $[vv_4v_5]$, the remaining two 3-faces must be one of the following:
(i) $[v_3v_7v_4]$ and $[v_4v_8v_5]$, or 
(ii) $[v_4v_7v_8]$ and $[v_4v_8v_5]$.
In case (i), we construct $G' = G - v_4 + \{v_7v_8\}$.
In case (ii), we construct $G' = G - v_4 + \{v_3v_7\}$.
The graph $G'$ is proper with respect to $G$.
By the minimality of $G$, $G'$ has a 2-distance 20-coloring $\phi'$.
Let $\phi$ be a coloring of $G$ such that every vertex in $V(G)$, except for $v_4$, is colored using $\phi'$.
Since $|C_{\phi}(v_4)| \leq 19$ in each case, we can color $v_4$ with a safe color, a contradiction.

Next, we consider Cases~1, 2, 5, and 9.
We prove that $m_3(v_6) \leq 3$ in these cases.
Assume that $v_6$ is incident to four 3-faces.
It is clear that $G' = G - v_6 + \{vx\}$ is proper with respect to $G$.
Since $d_2(v_6) \leq 19$ in each case,  there exists a safe color for $v_6$, a contradiction.

We can similarly show that $m_3(v_3) \leq 3$ in Case~8.
Assume that $v_3$ is incident to four 3-faces.
Let $v_7$ and $v_8$ be the neighbours of $v_3$ other than $v$, $v_2$, and $v_4$.
Since $v_3$ is already incident to two 3-faces, namely $[vv_2v_3]$ and $[vv_3v_4]$, the remaining two 3-faces must be one of the following:
(i) $[v_2v_7v_3]$ and $[v_3v_8v_4]$, or 
(ii) $[v_3v_7v_8]$ and $[v_3v_8v_4]$.
In case (i), we construct $G' = G - v_3 + \{v_7v_8\}$.
In case (ii), we construct $G' = G - v_3 + \{v_2v_7\}$.
The graph $G'$ is proper with respect to $G$.
Since $d_2(v_3) \leq 19$ in each case,  there exists a safe color for $v_3$, a contradiction.

Finally, we discuss Case~6 and Case~10.
To show that there exist at most two 5-vertices $u$ adjacent to $v$ with $m_3(v) \geq 4$, it suffices to prove that at most two edges in $G[N_G(v)]$ are contained in two 3-faces of $G$.
Assume that three edges in $G[N_G(v)]$ are contained in two 3-faces. 
In each case, we construct $G' = G - v + \{v_2v_4, v_4v_6, v_6v_1\}$.
The graph $G'$ is proper with respect to $G$.
Since $d_2(v) \leq 19$, there exists a safe color for $v$, a contradiction.
From the above, there are at most two 5-vertices $u$ adjacent to $v$ with $m_3(u) \geq 4$ in Case~1 through Case~10.
\end{proof}

\begin{lem}\label{lem10}%Lemma10
Let $v$ be a 6-vertex with $m_3(v) = 5$ and $m_{5^+}(v) = 1$ and let $u$ be any 5-vertex adjacent to $v$.
Then the following hold:
	\begin{enumerate}
	\item[(1)] $n_4(v) \leq 2$.
	\item[(2)] If $n_4(v) = 2$, then $n_5(v) \leq 2$.
	\item[(3)] If $n_4(v) = 1$, then $n_5(v) \leq 4$. 
	Moreover, if $n_5(v) = 4$, then $m_3(u) \leq 3$.
	\item[(4)] If $n_4(v) = 0$ and $n_5(v) = 6$, then $m_3(u) \leq 3$.
	\item[(5)] If $n_4(v) = 0$ and $n_5(v) = 5$, then there exist at most two 5-vertices $u$ with $m_3(u) \geq 4$.
	\end{enumerate}
\end{lem}
\begin{proof}
Suppose that $v$ is incident to five 3-faces $[vv_1v_2]$, $[vv_2v_3]$, $[vv_3v_4]$, $[vv_4v_5]$, $[vv_5v_6]$, and one $5^+$-face that contains $v_1$ and $v_6$.
By Lemma~\ref{lem2}(2), $v$ is not adjacent to any 3-vertex.

(1) The proof is the same as that of Lemma~\ref{lem9}(1).
The vertices $v_1$ and $v_6$ can be 4-vertices.

(2) We suppose that $v_1$ and $v_6$ are 4-vertices.
Assume that $v$ is adjacent to three 5-vertices.
Regardless of which neighbours of $v$ other than $v_1$ and $v_6$ are the three 5-vertices, we construct $G' = G - v + \{v_1v_3, v_1v_5, v_1v_6\}$.
The graph $G'$ is proper with respect to $G$.
By the minimality of $G$, $G'$ has a 2-distance 20-coloring $\phi'$.
Let $\phi$ be a coloring of $G$ such that every vertex in $V(G)$, except for $v$, is colored using $\phi'$.
Since $\Delta \leq 6$, it follows that $|C_\phi(v)| \leq 19$ and $|C| - |C_\phi(v)| \geq 1$.
Therefore, there exists a safe color for $v$.
By coloring $v$ with the safe color, $\phi$ becomes a 2-distance 20-coloring of $G$, a contradiction.

(3) Without loss of generality, we suppose that $v_1$ is a 4-vertex.
Assume that $v$ is adjacent to five 5-vertices.
It is clear that $G' = G - v + \{v_1v_3, v_1v_5, v_1v_6\}$ is proper with respect to $G$.
By the minimality of $G$, $G'$ has a 2-distance 20-coloring $\phi'$.
Let $\phi$ be a coloring of $G$ such that every vertex in $V(G)$, except for $v$, is colored using $\phi'$.
Since $|C_{\phi}(v)| \leq 19$, we can color $v$ with a safe color, a contradiction.
Now, we suppose that $v$ is adjacent to four 5-vertices.
To show that any 5-vertex $u$ adjacent to $v$ cannot be incident to more than four 3-faces, it suffices to prove that no edge in $G[N_G(v)]$ is contained in two 3-faces of $G$.
Assume that the edge $v_1v_2$ is contained in two 3-faces.
Let $G' = G - v + \{v_1v_3, v_1v_5, v_1v_6\}$.
Then $G'$ is proper with respect to $G$.
Since $d_2(v) \leq 19$, there exists a safe color for $v$, a contradiction.

(4) Suppose that all neighbours of $v$ are 5-vertices.
To show that any 5-vertex $u$ adjacent to $v$ is incident to at most three 3-faces, it suffices to prove that no edge in $G[N_G(v)]$ is contained in two 3-faces of $G$.
Assume that the edge $v_1v_2$ is contained in two 3-faces.
Let $G' = G - v + \{v_4v_2, v_4v_6, v_1v_6\}$.
Then $G'$ is proper with respect to $G$.
Since $d_2(v) \leq 19$, there exists a safe color for $v$, a contradiction.

(5) Suppose that $v$ is adjacent to five 5-vertices and one 6-vertex.
To show that there exist at most two 5-vertices $u$ adjacent to $v$ with $m_3(u) \geq 4$, it suffices to prove that at most one edge in $G[N_G(v)]$ is contained in two 3-faces of $G$.
Assume that two edges $v_iv_{i+1}$ and $v_jv_{j+1}$ for $i, j \in \{1,2,3,4,5\}$ with $i \neq j$ are contained in two 3-faces.  
Due to symmetry, it suffices to consider the cases where $v_1$, $v_2$, or $v_3$ is a 6-vertex.
In each of these cases, let $G' = G - v + \{v_4v_2, v_4v_6, v_1v_6\}$.
Then $G'$ is proper with respect to $G$.
Since $d_2(v) \leq 19$, there exists a safe color for $v$, a contradiction.
\end{proof}

From Lemma~\ref{lem11} to Lemma~\ref{lem13}, let $f_i = [vv_iv_{i+1}]$ for $i \in \{1, 2, \ldots, 5\}$ and $f_6 = [vv_6v_1]$ be the 3-faces incident to $v$.

\begin{lem}\label{lem11}%Lemma11
Let $v$ be a 6-vertex with $m_3(v) = 4$ and $m_4(v) = 2$.
Then the following hold:
	\begin{enumerate}
	\item[(1)] $n_3(v) = 0$.
	\item[(2)] If $n_4(v) = 1$, then $n_5(v) \leq 4$.
	\item[(3)] If $n_5(v) = 6$, then $m_3(u) \leq 3$ for any 5-vertex $u$ adjacent to $v$.
	\end{enumerate}
\end{lem}
\begin{proof}
We have three cases where $v$ is incident to four 3-faces and two 4-faces:
Case~1: The 4-faces are $[vv_1xv_2]$ and $[vv_2yv_3]$, and the 3-faces are $f_3$, $f_4$, $f_5$, and $f_6$.
Case~2: The 4-faces are $[vv_1xv_2]$ and $[vv_3yv_4]$, and the 3-faces are $f_2$, $f_4$, $f_5$, and $f_6$.
Case~3: The 4-faces are $[vv_1xv_2]$ and $[vv_4yv_5]$, and the 3-faces are $f_2$, $f_3$, $f_5$, and $f_6$.

(1) By Lemma~\ref{lem2}(2), a 3-vertex is not incident to any 3-face, and by Lemma~\ref{lem2}(3), a 3-vertex is incident to at most one 4-face.
Thus $v$ is not adjacent to any 3-vertex in each case.

(2) Suppose that $v$ is adjacent to one 4-vertex.
Assume that all other neighbours of $v$ are 5-vertices.
First, we consider Case~1.
In the proof of Lemma~\ref{lem3}, we showed that if $v$ is a 4-vertex with $m_3(v) = 2$, then no edge in $G[N_G(v)]$ is contained in two 3-faces.
Thus only $v_1$, $v_2$, or $v_3$ can be a 4-vertex.
In each case, let $G' = G - v + \{v_1v_2, v_2v_3, v_3v_5, v_5v_1\}$.
Then $G'$ is proper with respect to $G$.
By the minimality of $G$, $G'$ has a 2-distance 20-coloring $\phi'$.
Let $\phi$ be a coloring of $G$ such that every vertex in $V(G)$, except for $v$, is colored using $\phi'$.
Since $|C_\phi(v)| \leq 19$ and $|C| - |C_\phi(v)| \geq 1$, there exists a safe color for $v$.
By coloring $v$ with the safe color, $\phi$ becomes a 2-distance 20-coloring of $G$, a contradiction.

Next, we consider Case~2.
For the same reason, only $v_1$, $v_2$, $v_3$, or $v_4$ can be a 4-vertex.
In each case, let $G' = G - v + \{v_1v_2, v_3v_4, v_3v_5, v_5v_1\}$.
Then $G'$ is proper with respect to $G$.
Since $d_2(v) \leq 19$, there exists a safe color for $v$, a contradiction.

Finally, we consider Case~3.
Each neighbour of $v$ can be a 4-vertex.
By symmetry, it suffices to consider the cases where $v_1$ or $v_3$ is a 4-vertex.
In each case, let $G' = G - v + \{v_1v_2, v_4v_5, v_1v_3, v_3v_5\}$.
Then $G'$ is proper with respect to $G$.
Since $d_2(v) \leq 19$, there exists a safe color for $v$, a contradiction.

(3) Suppose that all neighbours of $v$ are 5-vertices.
To show that any 5-vertex $u$ adjacent to $v$ cannot be incident to more than four 3-faces, it suffices to prove that no edge in $G[N_G(v)]$ is contained in two 3-faces of $G$.
Assume that the edge $v_6v_1$ is contained in two 3-faces.
In Case~1, let $G' = G - v + \{v_1v_2, v_2v_3, v_3v_5, v_5v_1\}$.
In Case~2, let $G' = G - v + \{v_1v_2, v_3v_4, v_3v_5, v_5v_1\}$.
In Case~3, let $G' = G - v + \{v_1v_2, v_4v_5, v_1v_3, v_3v_5\}$.
In each case, $G'$ is proper with respect to $G$.
Since $d_2(v) \leq 19$, there exists a safe color for $v$, a contradiction.
\end{proof}

\begin{lem}\label{lem12}%Lemma12
Let $v$ be a 6-vertex with $m_3(v) = 4$, $m_4(v) = 1$, and $m_{5^+}(v) = 1$.
Then $n_3(v) \leq 1$.
In particular, if $n_3(v) = 1$, then $n_4(v) = 0$.
\end{lem}
\begin{proof}
We have three cases where $v$ is incident to four 3-faces, one 4-face, and one $5^+$-face:
Case~1: The 4-face is $[vv_1xv_2]$ and the $5^+$-face is $[vv_2y \ldots zv_3]$, and the 3-faces are $f_3$, $f_4$, $f_5$, and $f_6$.
Case~2: The 4-face is $[vv_1xv_2]$ and the $5^+$-face is $[vv_3y \ldots zv_4]$, and the 3-faces are $f_2$, $f_4$, $f_5$, and $f_6$.
Case~3: The 4-face is $[vv_1xv_2]$ and the $5^+$-face is $[vv_4y \ldots zv_5]$, and the 3-faces are $f_2$, $f_3$, $f_5$, and $f_6$.

First, we show that $v$ is adjacent to at most one 3-vertex.
By Lemma~\ref{lem2}(2), a 3-vertex is not incident to any 3-face.
Thus $v$ is not adjacent to any 3-vertex in Case~2 and Case~3.
In Case~1, only $v_2$ can be a 3-vertex. 
Thus $n_3(v) \leq 1$ holds.

Now we consider Case~1 and suppose that $v_2$ is a 3-vertex.
Assume that $v$ is adjacent to a 4-vertex.
In the proof of Lemma~\ref{lem3}, we showed that if $v$ is a 4-vertex with $m_3(v) = 2$, then no edge in $G[N_G(v)]$ is contained in two 3-faces.
Hence only $v_1$ or $v_3$ can be a 4-vertex.
If $v_1$ is a 4-vertex, then we construct $G' = G - v_2 + \{v_1y\}$.
Otherwise, we construct $G' = G - v_2 + \{v_3x, v_3y\}$.
In both cases, $G'$ is proper with respect to $G$.
By the minimality of $G$, $G'$ has a 2-distance 20-coloring $\phi'$.
Let $\phi$ be a coloring of $G$ such that every vertex in $V(G)$, except for $v_2$, is colored using $\phi'$.
Since $\Delta \leq 6$, it follows that $|C_\phi(v_2)| \leq 17$ and $|C| - |C_\phi(v_2)| \geq 3$.
Therefore, there exists a safe color for $v_2$.
By coloring $v_2$ with the safe color, $\phi$ becomes a 2-distance 20-coloring of $G$, a contradiction.
\end{proof}

\begin{lem}\label{lem13}%Lemma13
Let $v$ be a 6-vertex with $m_3(v) = 4$ and $m_{5^+}(v) = 2$.
Then $n_3(v) \leq 1$.
In particular, if $n_3(v) = 1$, then $n_4(v) = 0$.
\end{lem}
\begin{proof}
We have three cases where $v$ is incident to four 3-faces and two $5^+$-faces:
Case~1: The $5^+$-faces are $[vv_1x \ldots yv_2]$, $[vv_2z \ldots wv_3]$, and the 3-faces are $f_3$, $f_4$, $f_5$, and $f_6$.
Case~2: The $5^+$-faces are $[vv_1x \ldots yv_2]$, $[vv_3z \ldots wv_4]$, and the 3-faces are $f_2$, $f_4$, $f_5$, and $f_6$.
Case~3: The $5^+$-faces are $[vv_1x \ldots yv_2]$, $[vv_4z \ldots wv_5]$, and the 3-faces are $f_2$, $f_3$, $f_5$, and $f_6$.

The proof is similar to that of Lemma~\ref{lem12}.
Only $v_2$ can be a 3-vertex in Case~1.
Thus $n_3(v) \leq 1$ holds.
We suppose that $v_2$ is a 3-vertex and assume that $v$ is adjacent to a 4-vertex.
Only $v_1$ or $v_3$ can be a 4-vertex.
If $v_1$ is a 4-vertex, then we construct $G' = G - v_2 + \{v_1y, v_1z\}$.
Otherwise, we construct $G' = G - v_2 + \{v_3y, v_3z\}$.
In both cases, $G'$ is proper with respect to $G$.
By the minimality of $G$, $G'$ has a 2-distance 20-coloring $\phi'$.
Let $\phi$ be a coloring of $G$ such that every vertex in $V(G)$, except for $v_2$, is colored using $\phi'$.
Since $|C_{\phi}(v_2)| \leq 18$, we can color $v_2$ with a safe color, a contradiction.
\end{proof}

\begin{lem}\label{lem14}%Lemma14
Let $f$ be a 5-face of $G$.
Then there is at most one 3-vertex incident to $f$. 
In particular, if $f$ is incident to one 3-vertex, then $f$ is not incident to any 4-vertex.
\end{lem}
\begin{proof}
Let $[v_1v_2v_3v_4v_5]$ be a 5-face.
By Lemma~\ref{lem2}(1), a 3-vertex is not adjacent to any $5^-$-vertex. 
Assume that $v_1$ and $v_4$ are 3-vertices with $N_G(v_1) = \{v_2, v_5, v_6\}$ and $N_G(v_4) = \{v_3, v_5, v_7\}$.
Let $G' = G - v_1 + \{v_2v_4, v_4v_6\}$.
The graph $G'$ is proper with respect to $G$.
By the minimality of $G$, $G'$ has a 2-distance 20-coloring $\phi'$.
Let $\phi$ be a coloring of $G$ such that every vertex in $V(G)$, except for $v_1$, is colored using $\phi'$.
Since $\Delta \leq 6$, it follows that $|C_\phi(v_1)| \leq 18$ and $|C| - |C_\phi(v_1)| \geq 2$.
Therefore, there exists a safe color for $v_1$.
By coloring $v_1$ with the safe color, $\phi$ becomes a 2-distance 20-coloring of $G$, a contradiction.

Now, suppose that $v_1$ is a 3-vertex with $N_G(v_1) = \{v_2, v_5, v_6\}$.
Assume that $v_4$ is a 4-vertex.
It is clear that $G' = G - v_1 + \{v_2v_4, v_4v_6\}$ is proper with respect to $G$.
Let $\phi$ be a coloring of $G$ such that every vertex in $V(G)$, except for $v_1$, is colored using $\phi'$.
Since $|C_\phi(v_1)| \leq 18$, we can color $v_1$ with a safe color, a contradiction.
\end{proof}

We obtain the following corollary from Lemma~\ref{lem2}. 
\begin{cor}\label{cor15}%Corollary15
A $6^+$-face $f$ is incident to at most $\lfloor \frac{d(f)}{2} \rfloor$ 3-vertices.
\end{cor}

\section{Discharging}
In this section, we design discharging rules and complete the proof of Theorem~\ref{main}.
We can derive the following equation by Euler's formula $|V(G)| - |E(G)| + |F(G)| = 2$.
$$\sum_{v \in V(G)}(d_G(v)-4) +\sum_{f \in F(G)}(d(f) - 4) = -8.$$

We assign an initial charge $\mu (v) = d_G(v)-4$ to each vertex and $\mu(f) = d(f)-4$ to each face.
We design appropriate discharging rules and redistribute the charges of the vertices and faces according to those rules.
Let $\mu'(v)$ and $\mu'(f)$ denote the final charges of the vertices and faces, respectively, after the discharging process.
During the process, the sum of charges remains constant.
If $\mu'(v) \geq 0$ and $\mu'(f) \geq 0$, the following contradiction arises.
$$0 \leq \sum_{x \in V(G) \cup F(G)}\mu'(x) = \sum_{x \in V(G) \cup F(G)}\mu(x) = -8 < 0.$$
We design the following discharging rules, which are based on the rules in  \cite{deniz20232distance}.
\begin{enumerate}
\item[R1] Every 3-face receives $\frac{1}{3}$ from each of its incident vertices.
\item[R2] Every 3-vertex receives $\frac{1}{9}$ from each of its adjacent 6-vertices $w$ with $m_3(w) \leq 5$.
\item[R3] Every 3-vertex receives $\frac{1}{3}$ from each of its incident $5^+$-faces.
\item[R4] Every 4-vertex receives $\frac{1}{5}$ from each of its incident $5^+$-faces.
\item[R5] Every 4-vertex receives $\frac{1}{15}$ from each of its adjacent 6-vertices $w$ with $m_3(w) \leq 5$.
\item[R6] Every 5-vertex receives $\frac{1}{5}$ from each of its incident $5^+$-faces.
\item[R7] Every 5-vertex $u$ with $m_3(u) \geq 4$ receives $\frac{2}{15}$ from each of its adjacent 6-vertices $w$ with $m_3(w) \leq 5$.
\item[R8] Every 6-vertex $v$ receives $\frac{1}{5}$ from each of its incident $5^+$-faces $f$, if $f$ does not contain any 3-vertex adjacent to $v$.
\item[R9] Every 6-vertex $v$ receives $\frac{1}{9}$ from each of its incident $5^+$-faces $f$, if $f$ contains a 3-vertex adjacent to $v$.
\end{enumerate}

First, we prove that $\mu'(f) \geq 0$ for each $f \in F(G)$.
\begin{description}
\item [Case 1. $d(f) = 3.$]\mbox{}\\
The initial charge is $\mu(f) = d(f) - 4 = -1$.
By Lemma~\ref{lem2}(2), $f$ is not incident to any 3-vertex.
By R1, $f$ receives $\frac{1}{3}$ from each $4^+$-vertex incident to $f$.
Thus $\mu'(f) = -1 + 3 \times \frac{1}{3} = 0$.
\item [Case 2. $d(f) = 4.$]\mbox{}\\
By the discharging rules, there is no transfer of charge.
Thus $\mu(f) = \mu'(f) = 0$.
\item [Case 3. $d(f) = 5.$]\mbox{}\\
The initial charge is $\mu(f) = d(f) - 4 = 1$.
By Lemma~\ref{lem14}, $f$ is incident to at most one 3-vertex, and if $f$ is incident to one 3-vertex, then $f$ is not incident to any 4-vertex.
By Lemma~\ref{lem2}(1), all neighbours of a 3-vertex are 6-vertices.
If $f$ is incident to a 3-vertex, then $\mu'(f) = 1 - \frac{1}{3} - 2 \times \frac{1}{5} - 2 \times \frac{1}{9} = \frac{2}{45}$ by R3, R6, R8, and R9.
Otherwise, $\mu'(f) = 1 - 5 \times \frac{1}{5} = 0$ by R4, R6, and R8.
\item [Case 4. $d(f) = 6^+.$]\mbox{}\\
The initial charge is $\mu(f) = d(f) - 4 \geq 2$. 
We have $\mu'(f) \geq 0$ by By Corollary~\ref{cor15}.
\end{description}

Next, we prove that $\mu'(v) \geq 0$ for each $v \in V(G)$.
By Lemma~\ref{lem1} and $\Delta \leq 6$, we only consider the cases where $3 \leq d_G(v) \leq 6$.
\begin{description}
\item [Case 1. $d_G(v) = 3.$]\mbox{}\\
The initial charge is $\mu(v) = d_G(v) - 4 = -1$.
By Lemma~\ref{lem2}(1), all neighbours of $v$ are 6-vertices.
By Lemma~\ref{lem2}(2) and Lemma~\ref{lem2}(3), $v$ is incident to either one 4-face and two $5^+$-faces or three $5^+$-faces.
In each case, $m_3(w) \leq 5$ holds for any 6-vertex $w$ adjacent to $v$.
If $v$ is incident to one 4-face and two $5^+$-faces, then $\mu'(v) = -1 + 3 \times \frac{1}{9} + 2 \times \frac{1}{3} = 0$ by R2 and R3. 
Otherwise, $\mu'(v) = -1 + 3 \times \frac{1}{9} + 3 \times \frac{1}{3} = \frac{1}{3}$ by R2 and R3. 

\item [Case 2. $d_G(v) = 4.$]\mbox{}\\
The initial charge is $\mu(v) = d_G(v) - 4 = 0$.
By Lemma~\ref{lem2}(1), $v$ is not adjacent to any 3-vertex.
By Lemma~\ref{lem3}, we have $m_3(v) \leq 2$.
Thus we divide the case based on the value of $m_3(v)$.

\item [Case 2.1. $m_3(v) = 2.$]\mbox{}\\
By Lemma~\ref{lem3}, we have $m_4(v) = 0$, $n_6(v) = 4$, and $m_3(w) \leq 4$ for any 6-vertex $w$ adjacent to $v$.
This implies that $v$ is incident to two 3-faces and two $5^+$-faces, all neighbours of $v$ are 6-vertices, and R5 can be applied to $v$.
By R1, R4, and R5, $\mu'(v) = 0 - 2 \times \frac{1}{3} + 2 \times \frac{1}{5} + 4 \times \frac{1}{15} = 0$.

\item [Case 2.2. $m_3(v) = 1.$]\mbox{}\\
By Lemma~\ref{lem4}, we have $m_4(v) \leq 2$.
In particular, if $1 \leq m_4(v) \leq 2$, then $n_4(v) = 0$ and $n_5(v) \leq 1$.
If $m_4(v) = 2$, then the remaining face incident to $v$ is a $5^+$-face, and $v$ is adjacent to at least three 6-vertices.
By R1, R4, and R5, $\mu'(v) \geq 0 - 1 \times \frac{1}{3} + 1 \times \frac{1}{5} + 3 \times \frac{1}{15} = \frac{1}{15}$.
If $m_4(v) = 1$, then the remaining two faces incident to $v$ are $5^+$-faces, and $v$ is adjacent to at least three 6-vertices.
By R1, R4, and R5, $\mu'(v) \geq 0 - 1 \times \frac{1}{3} + 2 \times \frac{1}{5} + 3 \times \frac{1}{15} = \frac{4}{15}$.
If $m_4(v) = 0$, then the remaining three faces incident to $v$ are $5^+$-faces.
Regardless of the number of 6-vertices adjacent to $v$, we have $\mu'(v) \geq 0 - 1 \times \frac{1}{3} + 3 \times \frac{1}{5} = \frac{4}{15}$ by R1 and R4.

\item [Case 2.3. $m_3(v) = 0.$]\mbox{}\\
In this case, $v$ is not incident to any 3-face, which implies that R1 cannot be applied.
Thus $\mu'(v) \geq \mu(v) = 0$.

\item [Case 3. $d_G(v) = 5.$]\mbox{}\\
The initial charge is $\mu(v) = d_G(v) - 4 = 1$.
By Lemma~\ref{lem2}(1), $v$ is not adjacent to any 3-vertex.
We divide the case based on the value of $m_3(v)$.

\item [Case 3.1. $m_3(v) = 5.$]\mbox{}\\
By Lemma~\ref{lem5}, we have $n_6(v) = 5$ and $m_3(w) \leq 4$ for any 6-vertex $w$ adjacent to $v$.
Thus R7 can be applied to $v$.
By R1 and R7, $\mu'(v) = 1 - 5 \times \frac{1}{3} + 5 \times \frac{2}{15} = 0$.

\item [Case 3.2. $m_3(v) = 4.$]\mbox{}\\
The remaining face incident to $v$ is either one 4-face or one $5^+$-face.
First, we consider the case where the remaining face is a 4-face.
By Lemma~\ref{lem6}, we have $n_{4^-}(v) = 0$, $n_5(v) \leq 1$, and $m_3(w) \leq 5$ for any 6-vertex $w$ adjacent to $v$.
Thus R7 can be applied to $v$.
If $v$ is adjacent to one 5-vertex and four 6-vertices, then $\mu'(v) \geq 1 - 4 \times \frac{1}{3} + 4 \times \frac{2}{15} = \frac{1}{5}$ by R1 and R7.
If $v$ is adjacent to five 6-vertices, then $\mu'(v) \geq 1 - 4 \times \frac{1}{3} + 5 \times \frac{2}{15} = \frac{1}{3}$ by R1 and R7.
Next, we consider the case where the remaining face is a $5^+$-face.
By Lemma~\ref{lem8}, the pattern of the degrees of the vertices adjacent to $v$ must be one of the cases (a) through (d). 
In each case, we show that $\mu'(v) \geq 0$.
	\begin{description}
	\item[(a).] ($n_4(v)$, $n_5(v)$, $n_6(v)$) = (1, 0, 4).\mbox{}\\
	By Lemma~\ref{lem8}(1), we have $m_3(w) \leq 4$ for any 6-vertex $w$ adjacent to $v$.
	Thus R7 can be applied to $v$.
	By R1, R6, and R7, $\mu'(v) = 1 - 4 \times \frac{1}{3} + 1 \times \frac{1}{5} + 4 \times \frac{2}{15} = \frac{2}{5}$.
	
	\item[(b).] ($n_4(v)$, $n_5(v)$, $n_6(v)$) = (0, 2, 3).\mbox{}\\
	By Lemma~\ref{lem8}(2), we have $m_3(w) \leq 4$ for any 6-vertex $w$ adjacent to $v$.
	Thus R7 can be applied to $v$.
	By R1, R6, and R7, $\mu'(v) = 1 - 4 \times \frac{1}{3} + 1 \times \frac{1}{5} + 3 \times \frac{2}{15}  = \frac{4}{15}$.
	
	\item[(c).] ($n_4(v)$, $n_5(v)$, $n_6(v)$) = (0, 1, 4).\mbox{}\\
	By Lemma~\ref{lem8}(3), there exists at least one 6-vertex $w$ adjacent to $v$ with $m_3(w) \leq 5$.
	Thus $v$ receives at least $\frac{2}{15}$ from such a 6-vertex by R7.
	By R1, R6, and R7, $\mu'(v) \geq 1 - 4 \times \frac{1}{3} + 1 \times \frac{1}{5} + 1 \times \frac{2}{15} = 0$.
	
	\item[(d).] ($n_4(v)$, $n_5(v)$, $n_6(v)$) = (0, 0, 5).\mbox{}\\
	By Lemma~\ref{lem8}(4), there exist at least two 6-vertices $w_1$, $w_2$ adjacent to $v$ with $m_3(w_1) \leq 5$ and $m_3(w_2) \leq 5$.
	Thus $v$ receives at least $2 \times \frac{2}{15}$ from such 6-vertices by R7.
	By R1, R6, and R7, $\mu'(v) \geq 1 - 4 \times \frac{1}{3} + 1 \times \frac{1}{5} + 2 \times \frac{2}{15} = \frac{2}{15}$.
	\end{description}
	
\item [Case 3.3. $m_3(v) \leq 3.$]\mbox{}\\
The only rule by which $v$ loses charge is R1.
By R1 and $m_3(v) \leq 3$, we have $\mu'(v) \geq 1 - 3 \times \frac{1}{3} = 0$.

\item [Case 4. $d_G(v) = 6.$]\mbox{}\\
The initial charge is $\mu(v) = d_G(v) - 4 = 2$.
We divide the case based on the value of $m_3(v)$.

\item [Case 4.1. $m_3(v) = 6.$]\mbox{}\\
Since $m_3(v) = 6$, R2, R5, and R7 cannot be applied to $v$.
The only rule by which $v$ loses charge is R1.
Thus $\mu'(v) = 2 - 6 \times \frac{1}{3} = 0$.

\item [Case 4.2. $m_3(v) = 5.$]\mbox{}\\
By R1, $v$ sends $\frac{1}{3}$ to each of its incident 3-faces.
Since $m_3(v) = 5$, $v$ loses $5 \times \frac{1}{3} = \frac{5}{3}$ charge.
By Lemma~\ref{lem2}(2), $v$ is not adjacent to any 3-vertex.
The remaining face incident to $v$ is either one 4-face or one $5^+$-face.
First, we consider the case where the remaining face is a 4-face.
By Lemma~\ref{lem9}(1), we have $n_4(v) \leq 2$.
Thus we further divide the case based on the value of $n_4(v)$.
	\begin{description}
	\item[Case 4.2.1. $m_4(v) = 1$, $n_4(v) = 2.$]\mbox{}\\
	By Lemma~\ref{lem9}(2), we have $n_5(v) \leq 1$.
	In the worst situation, $v$ is adjacent to one 5-vertex $u$ with $m_3(u) \geq 4$.
	By R1, R5, and R7, $\mu'(v) \geq 2 - \frac{5}{3} - 2 \times \frac{1}{15} - 1 \times \frac{2}{15} = \frac{1}{15}$.
	
	\item[Case 4.2.2. $m_4(v) = 1$, $n_4(v) = 1.$]\mbox{}\\
	By Lemma~\ref{lem9}(3), we have $n_5(v) \leq 3$, and if $n_5(v) = 3$, then $m_3(u) \leq 3$ for any 5-vertex $u$ adjacent to $v$.
	Thus if $n_5(v) = 3$, then $v$ does not lose charge by R7.
	The vertex $v$ loses the most charge when $v$ is adjacent to two 5-vertices to which R7 applies.
	By R1, R5, and R7, $\mu'(v) \geq 2 - \frac{5}{3} - 1 \times \frac{1}{15} - 2 \times \frac{2}{15} = 0$.
	
	\item[Case 4.2.3. $m_4(v) = 1$, $n_4(v) = 0.$]\mbox{}\\
	By Lemma~\ref{lem9}(4) and Lemma~\ref{lem9}(4.1), we have $n_5(v) \leq 5$, and if $n_5(v) = 5$, then $m_3(u) \leq 3$ for any 5-vertex $u$ adjacent to $v$.
	Thus if $n_5(v) = 5$, then $v$ does not lose charge by R7.
	If $n_5(v) \leq 4$, then there exist at most two 5-vertices $u$ with $m_3(u) \geq 4$ by Lemma~\ref{lem9}(4.2) and Lemma~\ref{lem9}(4.3).
	This implies that $v$ loses at most $2 \times \frac{2}{15}$ charge by R7.
	By R1 and R7, $\mu'(v) \geq 2 - \frac{5}{3} - 2 \times \frac{2}{15} = \frac{1}{15}$.
	\end{description}
	
	Next, we consider the case where the remaining face is a $5^+$-face.
	By Lemma~\ref{lem10}(1), we have $n_4(v) \leq 2$.
	Thus we further divide the case based on the value of $n_4(v)$. 	
	\begin{description}
	\item[Case 4.2.4. $m_{5^+}(v) = 1$, $n_4(v) = 2.$]\mbox{}\\
	By Lemma~\ref{lem10}(2), we have $n_5(v) \leq 2$.
	This implies that $v$ loses at most $2 \times \frac{2}{15}$ charge by R7.
	By R1, R5, R7, and R8, $\mu'(v) \geq 2 - \frac{5}{3} - 2 \times \frac{1}{15} - 2 \times \frac{2}{15} + \frac{1}{5} = \frac{2}{15}$.
	
	\item[Case 4.2.5. $m_{5^+}(v) = 1$, $n_4(v) = 1.$]\mbox{}\\
	By Lemma~\ref{lem10}(3), we have $n_5(v) \leq 4$, and if $n_5(v) = 4$, then $m_3(u) \leq 3$ for any 5-vertex $u$ adjacent to $v$.
	Thus if $n_5(v) = 4$, then $v$ does not lose charge by R7.
	The vertex $v$ loses the most charge when $v$ is adjacent to three 5-vertices to which R7 applies.
	By R1, R5, R7, and R8, $\mu'(v) \geq 2 - \frac{5}{3} - 1 \times \frac{1}{15} - 3 \times \frac{2}{15} + \frac{1}{5} = \frac{1}{15}$.
	
	\item[Case 4.2.6. $m_{5^+}(v) = 1$, $n_4(v) = 0.$]\mbox{}\\
	By Lemma~\ref{lem10}(4), if $n_5(v) = 6$, then $m_3(u) \leq 3$ for any 5-vertex $u$ adjacent to $v$.
	Hence if $n_5(v) = 6$, then $v$ does not lose charge by R7.
	By Lemma~\ref{lem10}(5), if $n_5(v) = 5$, then there exist at most two 5-vertices $u$ with $m_3(u) \geq 4$.
	Thus $v$ loses the most charge when $v$ is adjacent to four 5-vertices to which R7 applies.
	By R1, R7, and R8, $\mu'(v) \geq 2 - \frac{5}{3} - 4 \times \frac{2}{15} + \frac{1}{5} = 0$.
	
	\end{description}
\item [Case 4.3. $m_3(v) = 4.$]\mbox{}\\
By R1, $v$ sends $\frac{1}{3}$ to each of its incident 3-faces.
Since $m_3(v) = 4$, $v$ loses $4 \times \frac{1}{3} = \frac{4}{3}$ charge.
We further divide the case based on the faces incident to $v$, which can be either two 4-faces, one 4-face and one $5^+$-face, or two $5^+$-faces.
	\begin{description}
	\item [Case 4.3.1. $m_4(v) = 2.$]\mbox{}\\
	By Lemma~\ref{lem11}(1), $v$ is not adjacent to any 3-vertex.
	The rule by which $v$ loses the most charge, except for R1, is R7.
	Thus if all neighbours of $v$ are 5-vertices to which R7 applies, then $v$ loses $6 \times \frac{2}{15} = \frac{4}{5}$ by R7.
	This discussion implies that $\mu'(v) = 2 - \frac{4}{3} - \frac{4}{5} = -\frac{2}{15} < 0$.
	However, by Lemma~\ref{lem11}(3), the situation does not arise.
	The next situation in which $v$ loses the most charge is when $v$ is adjacent to five 5-vertices $u$ with $m_3(u) \geq 4$ and one 4-vertex, but by Lemma~\ref{lem11}(2), this situation cannot occur.
	In the possible cases, $v$ loses the most charge when $v$ is adjacent to either four 5-vertices $u$ with $m_3(u) \geq 4$ and two 4-vertices, or five 5-vertices $u$ with $m_3(u) \geq 4$ and one 6-vertex.
	In the former case, $\mu'(v) = 2 - \frac{4}{3} - 2 \times \frac{1}{15} - 4 \times \frac{2}{15} = 0$ by R1, R5, and R7.
	In the latter case, $\mu'(v) = 2 - \frac{4}{3} - 5 \times \frac{2}{15} = 0$ by R1 and R7.
	
	\item [Case 4.3.2. $m_4(v) = 1, m_{5^+}(v) = 1.$]\mbox{}\\
	By Lemma~\ref{lem12}, $v$ is adjacent to at most one 3-vertex, and if $v$ is adjacent to one 3-vertex, then $v$ is not adjacent to any 4-vertex.
	Since $v$ is incident to one $5^+$-face, R8 or R9 can be applied to $v$.
	If $v$ is adjacent to one 3-vertex, then $v$ loses the most charge when $v$ is adjacent to five 5-vertices to which R7 applies.
	By R1, R2, R7, and R9, $\mu'(v) = 2 - \frac{4}{3} - 1 \times \frac{1}{9} - 5 \times \frac{2}{15} + 1 \times \frac{1}{9} = 0$.
	If $v$ is not adjacent to a 3-vertex, then $v$ loses the most charge when $v$ is adjacent to six 5-vertices to which R7 applies.
	By R1, R7, and R8, $\mu'(v) = 2 - \frac{4}{3} - 6 \times \frac{2}{15} + 1 \times \frac{1}{5} = \frac{1}{15}$.
	
	\item [Case 4.3.3. $m_{5^+}(v) = 2.$]\mbox{}\\
	By Lemma~\ref{lem13}, $v$ is adjacent to at most one 3-vertex, and if $v$ is adjacent to one 3-vertex, then $v$ is not adjacent to any 4-vertex.
	Since $v$ is incident to two $5^+$-faces, $v$ receives at least $2 \times \frac{1}{9}$ charge by R9.
	If $v$ is adjacent to one 3-vertex, then $v$ loses the most charge when $v$ is adjacent to five 5-vertices to which R7 applies.
	By R1, R2, R7, and R9, $\mu'(v) = 2 - \frac{4}{3} - 1 \times \frac{1}{9} - 5 \times \frac{2}{15} + 2 \times \frac{1}{9} = \frac{1}{9}$.
	If $v$ is not adjacent to a 3-vertex, then $v$ loses the most charge when $v$ is adjacent to six 5-vertices to which R7 applies.
	By R1, R7, and R8, $\mu'(v) = 2 - \frac{4}{3} - 6 \times \frac{2}{15} + 2 \times \frac{1}{5} = \frac{4}{15}$.
	\end{description}
	
\item [Case 4.4. $m_3(v) \leq 3.$]\mbox{}\\
By R1, $v$ loses at most $3 \times \frac{1}{3} = 1$ charge.
The rule by which $v$ loses the most charge, except for R1, is R7.
Thus if all neighbours of $v$ are 5-vertices $u$ with $m_3(u) \geq 4$, then $v$ loses $6 \times \frac{2}{15} = \frac{4}{5}$ charge by R7.
The final charge is $\mu'(v) \geq 2 - 1 - \frac{4}{5} = \frac{1}{5} > 0$, which implies that $\mu'(v) \geq 0$ holds when $m_3(v) \leq 3$.
\end{description}
Now, we have confirmed $\mu'(x) \geq 0$ for all $x \in V(G) \cup F(G)$, which is a contradiction.
Therefore, Theorem~\ref{main} holds.

%\section*{Declaration of competing interest}
%The authors declare that they have no known competing financial interests or personal relationships that could have appeared to influence the work reported in this paper.

%\section*{Funding}
%This research did not receive any specific grant from funding agencies in the public, commercial, or not-for-profit sectors.
%\section*{Acknowledgements}

\bibliographystyle{abbrv} 
\bibliography{ref}

\end{document}